\pdfoutput=1
\documentclass[letterpaper, 10 pt, conference]{ieeeconf}  % Comment this line out if you need a4paper

\IEEEoverridecommandlockouts                              % This command is only needed if 
\usepackage{mathrsfs}
\usepackage[font={small}]{caption}
\usepackage{scalerel}
\usepackage{url}
\usepackage{bbold}
\usepackage{amsfonts}
\usepackage{amsmath,amssymb,amsfonts}
\usepackage{graphicx}
\usepackage{wrapfig}

\usepackage{mathrsfs} 
\usepackage{algorithm,algorithmic}
\usepackage{array}
\usepackage{times}
\usepackage{url}
\usepackage{subfigure}
\usepackage{cite}
\usepackage{upgreek}
\usepackage{float}
\usepackage{longtable}
\usepackage{color}
\usepackage{wasysym}
\usepackage{grffile}
\usepackage{stackengine}

\usepackage{mathtools}
\usepackage{tikz}

\allowdisplaybreaks

\newcommand{\0}{\mathbb{0}}
\newcommand{\1}{\mathbb{1}}

\newcommand{\R}{\mathbb{R}}

\newcommand{\bb}{\mathbf{b}}

\newcommand{\wb}{\mathbf{w}}
\newcommand{\xb}{\mathbf{x}}
\newcommand{\yb}{\mathbf{y}}
\newcommand{\zb}{\mathbf{z}}

\newcommand{\Ab}{\mathbf{A}}

\newcommand{\Lb}{\mathbf{L}}

\newcommand{\Qb}{\mathbf{Q}}

\newcommand{\Sb}{\mathbf{S}}

\newcommand{\Dc}{\mathcal{D}}
\newcommand{\Ec}{\mathcal{E}}

\newcommand{\Gc}{\mathcal{G}}

\newcommand{\Lc}{\mathcal{L}}
\newcommand{\Nc}{\mathcal{N}}

\newcommand{\Vc}{\mathcal{V}}

\newcommand{\argmin}{\text{argmin}}

\newcommand{\diag}{\text{diag}}
\newcommand{\col}{\text{col}}

\newcommand{\norm}[1]{\left\lVert#1\right\rVert}

\newtheorem{theorem}{Theorem}

\newtheorem{assumption}{Assumption}

\newtheorem{lemma}[theorem]{Lemma}

\title{\LARGE \bf Distributed Mirror Descent with Integral Feedback: Asymptotic Convergence Analysis of Continuous-time Dynamics}

\author{Youbang Sun and Shahin Shahrampour, {\it Senior Member}, {\it IEEE}    
\thanks{Y. Sun and S. Shahrampour are with Wm Michael Barnes '64 Department of Industrial and Systems Engineering, Texas A\&M University, College Station, TX 77843, USA. 
        {\tt\small email:\{ybsun,shahin\}@tamu.edu}.}%
}

\begin{document}

\maketitle
\thispagestyle{empty}
\pagestyle{empty}

%%%%%%%%%%%%%%%%%%%%%%%%%%%%%%%%%%%%%%%%%%%%%%%%%%%%%%%%%%%%%%%%%%%%%%%%%%%%%%%%
\begin{abstract}
This work addresses distributed optimization, where a network of agents wants to minimize a global strongly convex objective function. The global function can be written as a sum of local convex functions, each of which is associated with an agent. We propose a continuous-time distributed mirror descent algorithm that uses purely local information to converge to the global optimum. Unlike previous work on distributed mirror descent, we incorporate an integral feedback in the update, allowing the algorithm to converge with a constant step-size when discretized. We establish the asymptotic convergence of the algorithm using Lyapunov stability analysis. We further illustrate numerical experiments that verify the advantage of adopting integral feedback for improving the convergence rate of distributed mirror descent. 
\end{abstract}

%%%%%%%%%%%%%%%%%%%%%%%%%%%%%%%%%%%%%%%%%%%%%%%%%%%%%%%%%%%%%%%%%%%%%%%%%%%%%%%%
\section{Introduction}
The mirror descent (MD) algorithm \cite{nemirovsky1983originalmd} is a primal-dual method that has been successfully used for large-scale convex optimization problems. MD can be seen as a generalization of gradient descent, which can exploit the geometry of the optimization problem. The algorithm replaces the Euclidean distance with a so-called {\it Bregman divergence} as the regularizer for projection. This idea provides a significant convergence speed-up for high-dimensional optimization problems \cite{ben2001ordered}.

In practice, optimization methods (including MD) are numerically implemented in discrete time, but their continuous-time analysis has always been of major interest to the control and optimization community \cite{bloch1994hamiltonian,brown1989some,helmke2012optimization}. This stems from the fact that many optimization methods can be interpreted as discretization of ordinary differential equations (ODEs), and therefore, their convergence can be established using the theory of control and dynamical systems. The MD algorithm is no exception in this regard, and it can be studied via a system of ODEs \cite{krichene2015acceleratedmd}.

In this work, we address distributed continuous-time optimization via decentralized mirror descent, inspired by the success of centralized MD in large-scale optimization. In this setup, a network of agents wants to minimize a global strongly convex objective function. The global function can be written as a sum of local convex functions, each of which is associated with an agent. We develop a continuous-time decentralized MD algorithm that uses purely local gradient information to converge to the global minimizer. Contrary to the prior work on (discrete) distributed mirror descent (e.g., \cite{li2016distributed,shahrampour2017distributed}), we enforce consensus among agents using the idea of integral feedback, in addition to the standard neighborhood averaging. The integral feedback is particularly useful for implementation purposes, allowing the algorithm to converge with a constant step-size when discretized. We establish the asymptotic convergence using Lyapunov stability analysis, based on a Lyapunov function that relies on both primal and dual variables. Our numerical experiments verify that adopting integral feedback improves the convergence rate of distributed mirror descent.

\subsection{Related Literature}

{\bf I) Gradient Tracking in Discrete Distributed Gradient Descent (DGD):} A natural question in (discrete) distributed optimization is that whether decentralized algorithms are able to perform on par with their centralized counterparts. For purely convex problems (non-strongly convex and non-smooth), this could be done using diminishing step sizes \cite{nedic2009distributed}, which tends agents to an agreement. However, since centralized gradient descent for strongly convex and/or smooth problems works optimally under the constant step-size setting, its decentralization was challenging. Therefore, a number of works (see e.g., \cite{shi2015extra, sun2019convergence, di2016next}) have proposed the idea of gradient tracking to overcome this hurdle. The term ``tracking" implies that the algorithm uses a variable calculated from past gradients to keep track of the information from the network. It then uses the variable combined with the current local gradient to output a ``corrected" gradient such that the network agents are able to reach consensus. Such modification enables the decentralized algorithm to match its centralized counterpart in terms of convergence rate.

\textbf{II) Continuous-time DGD:} Of particular relevance to the current work is the literature on {\it continuous-time} DGD \cite{lin2019distributed, liu2014continuous, cortes2013weightbalanceddigraph, 7744584, kia2015distributed,yang2016multi}. Similar to the centralized setup, these works construct ODEs to describe the dynamics of DGD. While the concept of DGD is rather straightforward, for continuous-time analysis in certain cases (e.g., strongly convex problem), no desirable results are obtained by simply combining gradient descent with a standard neighborhood averaging. This is in the similar spirit as the challenge in discrete DGD, overcome by gradient tracking. To tackle the continuous-time problem, multiple works have utilized the integral feedback idea \cite{cortes2013weightbalanceddigraph, kia2015distributed, 7744584, yang2016multi}, which introduces another variable to drive the disagreement among agents to zero. Nevertheless, these works are on gradient descent, and investigating this idea for MD, which is a more general framework, is still an open research problem.

{\bf III) Distributed Mirror Descent (DMD):} Decentralizing mirror descent has recently drawn a great deal of attention. While our focus is on the {\it continuous-time} analysis, DMD has been largely analyzed in {\it discrete} time in various contexts, such as online optimization \cite{shahrampour2017distributed,9070199}, stochastic optimization \cite{yuan2018optimal,7383850}, and the
effect of delays in distributed optimization \cite{li2016distributed}. It has also been applied to social learning and belief dynamics \cite{shahrampour2015distributed}. Furthermore, Doan et al. \cite{8409957} study the convergence of iterates for both centralized and decentralized MD. A large subset of these works (e.g., \cite{li2016distributed,shahrampour2017distributed,9070199,8409957,yuan2018optimal}) feature diminishing step-size to ensure consensus. Continuous-time DMD has been studied in \cite{borovykh2020interact,raginsky2012Variancereduction} with the motivation of noise-variance reduction in stochastic optimization. The main distinction of our work with the literature on DMD is adopting and analyzing the idea of integral feedback.

\subsection{Paper Organization}
The rest of this paper is organized as follows. In Section \ref{Formulation}, we lay out the problem formulation and develop the continuous-time distributed mirror descent with integral feedback. In Section \ref{Results}, we provide the theoretical convergence analysis of the algorithm using Lyapunov stability analysis. Section \ref{simulation} provides a discretized version of our algorithm and illustrates a numerical simulation to show effectiveness of the proposed algorithm, and Section \ref{Conclusion} concludes.

\section{Problem Formulation}\label{Formulation}

\subsection{Notation}
We use the following notation in this paper:
% {\bl "Identity matrix" here is bold, make it consistent}
\begin{center}
\renewcommand{\arraystretch}{1.2} 
% table looks a bit cramped so I increased the vertical space
    \begin{tabular}{|c||l|}
    \hline
         $[n]$& set $\{1,2,3,\ldots,n\}$ for any integer $n$  \\
         \hline
         $x^\top$ & transpose of vector $x$ \\
         \hline
         $I_d$ & identity matrix of size $d\times d$\\
         \hline
         $\1_d$ & $d$-dimensional vector of all ones\\
          \hline
         $\0$ & vector of all zeros\\
         \hline
         $\norm{\cdot}$ & Euclidean norm operator\\
         \hline
         $\langle x, y \rangle$ & inner product between $x$ and $y$\\
         \hline
         $[x]_i $  & the $i$-th element of the vector $x$\\
         \hline
         $[A]_{ij}$ & the $ij$-th element of the matrix $A$\\
         \hline
         $A^\dagger$ & pseudo inverse of matrix $A$\\ 
         \hline
         $\otimes$ &  Kronecker product operator\\
         \hline     
    \end{tabular}
\end{center}
The vectors are all in column format. We denote by $\col\{v_1, \ldots, v_n\}$ the vector that stacks all vectors $v_i$ for $i\in [n]$. We use $\diag\{a_1, \ldots, a_n\}$ to represent an $n\times n$ diagonal matrix that has the scalar $a_i$ in its $i$-th diagonal element. 

\subsection{Network Setting}
In distributed optimization, we often consider a network of $n$ agents modeled with a graph $\Gc = (\Vc,\Ec)$, where the agents are represented by nodes $\Vc = [n]$ and the connection between two agents $i$ and $j$ is captured by the edge $\{i,j\} \in \Ec$. Each agent is associated with a local cost function, and if the link $\{i,j\} \in \Ec$ exists, that implies agents $i$ and $j$ can exchange information about their respective cost functions. Then, agent $j$ is in the neighborhood of agent $i$, denoted by $\Nc_i\triangleq\{j \in \Vc: \{i,j\}\in \Ec\}$.

The agents work collectively to find the optimum of the global cost function, which is the sum of all cost local functions (to be defined precisely in Section \ref{sec:dist-opt}).  
\begin{assumption}\label{A1}
We assume the graph $\Gc$ is undirected and connected, i.e., there exists a path between any two distinct agents $i,j \in \Vc$.
We use $\Lc \in \R^{n\times n}$ to the represent the Laplacian of the graph $\Gc$.
\end{assumption}

The connectivity assumption implies that $\Lc$ has a unique null eigenvalue. That is, $\Lc\1_n=\0$, and $\1_n$ is the only direction (eigenvector) recovering the null eigenvalue.

\subsection{Distributed Optimization Problem}\label{sec:dist-opt}
In this paper, we consider a distributed (or decentralized) optimization problem in an unconstrained setting. Let us denote by $f_i:\mathbb{R}^d\rightarrow \mathbb{R}$, the cost function associated with agent $i\in [n]$. Then, the goal is to find the optimal solution of the global cost function $F$, which can be written as a sum of local cost functions as follows,
\begin{equation}\label{globalfunction}
   \underset{x \in \mathbb{R}^d}{\text{minimize}} \:\:\:\:\:\:F(x) = \sum_{i=1}^n f_i(x).
\end{equation}
The above formulation is equivalent to 
\begin{equation}
    \begin{aligned}\label{distributed_problem}
    \underset{x_i \in \mathbb{R}^d}{\text{minimize}} \:\:\:\:\:\:& \sum_{i=1}^n f_i(x_i),\\
    \text{subject to} \:\:\:\:\:\:& x_1=x_2=\cdot\cdot\cdot=x_n.
\end{aligned}
\end{equation}
Since individual agents do not have knowledge of $F$, they cannot find the global solution on their own, and they must communicate with each other to augment their incomplete information with that of their neighborhood. 

\begin{assumption}\label{A3}
For any agent $i\in \Vc$, we assume that the local cost function $f_i:\R^d\to\R$ is convex and differentiable.
\end{assumption}

While Assumption \ref{A3} implies that the global objective function $F$ is also convex and differentiable, we impose an additional assumption on the global cost as follows. 

\begin{assumption}\label{A2}
The global function $F:\R^d\to\R$ is strongly convex. The optimal value denoted by $F^\star$ exists, and the unique solution that achieves $F^\star$ is denoted by $x^\star$.
\end{assumption} 

The assumption above will be used later in the analysis to prove the uniqueness of equilibrium for our proposed distributed continuous-time algorithm. 

\subsection{Centralized Mirror Descent}
Since the focus of this work is on the mirror descent algorithm, we provide some background on the centralized algorithm in this section, before developing the distributed algorithm in Section \ref{subsec:E}.

Gradient descent methods iteratively minimize a first order approximation of a function plus a Euclidean regularizer. Mirror descent generalizes this idea to a non-Euclidean setup by using the notion of {\it Bregman divergence}, which replaces the Euclidean distance as the regularizer. The Bregman divergence is defined with respect to a generating function $\phi:\R^d\to \R$, as follows
$$\Dc_{\phi} (x, x') \triangleq \phi(x) -\phi(x') - \langle \nabla\phi(x'), x-x'\rangle.$$
It can be immediately seen from above that the Bregman divergence is not generally symmetric, thereby it is not a distance. 

\begin{assumption}\label{Assumption_phi}
 The generating function $\phi$ is closed, differentiable and $\mu_\phi$-strongly convex.
\end{assumption}
The assumption above is standard. For example, $\phi(x) = \frac{1}{2}\norm{x}^2$ (the generator for the Euclidean distance), as well as the negative entropy function $\phi(x) =  \sum_{j=1}^d [x]_j \log([x]_j)$ (the generator for the Kullback–Leibler divergence) both satisfy the assumption \cite{shahrampour2017distributed}. 

In discrete time, the unconstrained mirror descent algorithm with learning rate $\eta$ is written as 
\begin{equation}\label{originalmd}
    \begin{aligned}
        x^{(k+1)} &= \underset{x}{\argmin} \bigg\{ F(x^{(k)}) + \eta \nabla F(x^{(k)})^\top (x - x^{(k)})\\
        &\qquad+ \Dc_\phi(x, x^{(k)})\bigg\},
    \end{aligned}
\end{equation}
where using the Euclidean distance in lieu of the Bregman divergence (i.e., $\Dc_\phi(x, x^{(k)})=\frac{1}{2}\|x-x^{(k)}\|^2$) reduces the algorithm to a gradient descent. 

For writing the continuous-time dynamics of mirror decent, an equivalent form of the update above is more convenient to use. This equivalent form is based on the {\it convex conjugate} or {\it Fenchel dual} of function $\phi$, which is denoted by $\phi^\star$ and defined as follows  
$$\phi^\star(z) \triangleq \underset{x \in \R^d}{\text{sup}}\{\langle x, z\rangle - \phi(x)\}.$$
The definition above entails the subsequent equivalence   
$$
z' = \nabla\phi(x') \Longleftrightarrow x' = \nabla\phi^\star(z'),
$$
and Assumption \ref{Assumption_phi} guarantees that $\phi^\star$ is $\mu_\phi^{-1}$-smooth. More details can be found in \cite{hiriart2012fundamentals}.

With the definition of $\phi^\star$ in place, the update \eqref{originalmd} can be rewritten in the following equivalent form

\begin{equation} \label{step_by_step_md}
    \begin{aligned}
        % z^{(k)} &= \nabla\phi(x^{(k)})\\
        z^{(k+1)} &= z^{(k)} - \eta\nabla F(x^{(k)})\\
        x^{(k+1)} &= \nabla\phi^\star(z^{(k+1)}).
        % &\qquad+ \Dc_\phi(x, x^{(k)})\},
    \end{aligned}
\end{equation}
Then, taking the learning rate $\eta$ to be infinitesimally small, the centralized mirror descent ODE takes the following form
\begin{equation}
   \begin{aligned}
    \Dot{z} &= -\nabla F(x),\\
    x &=\nabla\phi^\star(z),\\
    x(0) = x_0, z(0)&=z_0 \:\:\text{with}\:\: x_0 =\nabla\phi^\star(z_0),
\end{aligned} 
\end{equation}
which has been studied in \cite{krichene2015acceleratedmd} (Section 2.1). It is easy to see that when $\phi(x) = \frac{1}{2}\norm{x}^2$, since $\phi^\star(z) = \frac{1}{2}\norm{z}^2$, we have that $x = \nabla\phi^\star(z)=z$, and the mirror descent ODE reduces to the gradient descent ODE.

\subsection{Distributed Mirror Descent with Integral Feedback} \label{subsec:E}
We now develop the distributed version of mirror descent algorithm. Motivated by the use of {\it integral feedback} \cite{cortes2013weightbalanceddigraph, kia2015distributed, 7744584, yang2016multi} to enforce consensus among agents, we propose the following continuous-time algorithm

\begin{equation}\label{setup}
\begin{aligned}
    \dot{z_i} &= - \nabla f_i(x_i) + \sum_{j \in \Nc_i} (x_j - x_i) + \int_0^t \sum_{j \in \Nc_i} (x_j - x_i)\\
    x_i & =\nabla\phi^\star(z_i),\\
\end{aligned}
\end{equation}
initialized with $x_i(0) = x_{i0}, z_i(0) = z_{i0}$, where $x_{i0} =\nabla\phi^\star(z_{i0})$. 

The dual update $z_i$ for agent $i \in [n]$ uses only private gradient information. It also enforces the primal variables in the neighborhood of $i$ to get close to each other by using both a {\it consensus} term and an {\it integral feedback}. Then, the second update maps the variable $z_i$ back to the primal space using $\phi^\star$.

To analyze \eqref{setup}, it is more convenient to stack all the local vectors as follows
\begin{equation}\label{rename}
\begin{aligned}
    \xb &\triangleq \col\{x_1, x_2,\ldots,x_n\}\\
    \zb & \triangleq \col\{z_1, z_2,\ldots,z_n\},\\
\end{aligned}
\end{equation}
and define the following notation
\begin{equation}\label{rename2}
\begin{aligned}
    \Lb & \triangleq \Lc \otimes I_d\\
    \nabla\phi^\star(\zb) & \triangleq \col\{\nabla\phi^\star(z_1),\nabla\phi^\star(z_2),\ldots,\nabla\phi^\star(z_n)\}\\
    \nabla f(\xb) & \triangleq \col\{\nabla f_1(x_1), \nabla f_2(x_2),\ldots, \nabla f_n(x_n)\}.
\end{aligned}
\end{equation}
Then, by introducing the variable $\yb$ to replace the integral, the dynamics given in \eqref{setup} can be rewritten as follows,
\begin{equation}
    \begin{aligned}
    \Dot{\mathbf{z}} &= -(\nabla f(\mathbf{x}) + \mathbf{L} \mathbf{x} + \mathbf{y}),\label{setup_with_y}\\
    \Dot{\mathbf{y}} &=  \mathbf{L} \mathbf{x},\\
    \xb &= \nabla\phi^\star(\zb),
\end{aligned}
\end{equation}
where $\mathbf{y}\in \mathbb{R}^{nd}$ and $\mathbf{y}(0)=\0.$

\section{Main Results}\label{Results}
In this section, we establish the theoretical convergence of the distributed mirror descent algorithm with integral feedback, proposed in \eqref{setup}. We prove that all agents will converge asymptotically to the minimizer of the global function $F$, defined in \eqref{globalfunction}. First, in Section \ref{equilb_section}, we show that the unique equilibrium of \eqref{setup} for primal variables coincides with the minimizer of problem \eqref{globalfunction}, and then we provide the proof for the asymptotic convergence to the equilibrium in Section \ref{Asymptotic_Convergence_section}. 

\subsection{Equilibrium Analysis}\label{equilb_section}
\begin{lemma}\label{equilb_lemma}
Given Assumptions \ref{A1}-\ref{Assumption_phi}, an equilibrium point for the continuous-time dynamics \eqref{setup} exists, and it is unique. In the equilibrium, $x_i^\star = x^\star =\nabla\phi^\star(z_i^\star)$ for all $i \in [n]$, i.e., the equilibrium point has the consensus property, and at equilibrium, the primal variable for each agent is the solution to problem \eqref{globalfunction}.
\end{lemma}

\begin{proof}
Since the continuous-time dynamics \eqref{setup_with_y} is equivalent to \eqref{setup}, to prove Lemma \ref{equilb_lemma}, it is sufficient to show that there exists a unique point $(\xb^\star,\yb^\star,\zb^\star)$ satisfying equilibrium conditions for \eqref{setup_with_y}: 

-- To have $\dot{\yb}=\0$, we need $\xb^\star$ to be in the null space of $\Lb=\Lc \otimes I_d$, which together with the connectivity assumption (Assumption \ref{A1}), it implies that $\xb^\star = \1_n \otimes a$ for some vector $a\in \R^d$. Next, we show that indeed $a=x^\star$, where $x^\star$ is the minimizer of $F$.

-- To have $\dot{\zb}=\0$, we need 
\begin{align}\label{eq:111}
    \nabla f(\mathbf{x}^\star) + \mathbf{y}^\star=\0.
\end{align}
Due to the initialization $\yb(0)=\0$, we have that 
\begin{align}\label{eq:32}
    \yb(t)=\Lb\int_0^t\xb(\tau)d\tau,
\end{align}
which implies 
$$(\1_n\otimes I_d)^\top\yb(t)=(\1_n\otimes I_d)^\top\Lb\int_0^t\xb(\tau)d\tau=\0.$$
Therefore, $(\1_n\otimes I_d)^\top\yb^\star=\0$, and combining this with \eqref{eq:111}, we get
$$(\1_n\otimes I_d)^\top\nabla f(\mathbf{x}^\star)= \0 \Longrightarrow \sum_{i=1}^n \nabla f_i(a)=\nabla F(a)= \0.$$
Due to the strong convexity of $F$ in Assumption \ref{A2}, the minimizer $x^\star$ is unique, and therefore, $a=x^\star$. Hence, the following point is the unique equilibrium
\begin{align*}
\xb^\star &= \1_n \otimes x^\star\\
\yb^\star &= - \nabla f(\xb^\star)\\
\zb^\star &= \nabla \phi (\xb^\star),
\end{align*}
thereby completing the proof.
\end{proof}

Note that though agents reach consensus at the global minimizer of $F$, since local objective functions $\{f_i\}_{i=1}^n$ do not have the same minimizers, $\nabla f(\xb^\star)$ is not necessarily zero, and more specifically, as proved in the lemma, we have the following relationship
\begin{equation}\label{eq:22}
    \nabla f(\mathbf{x}^\star) = - \mathbf{y}^\star.
\end{equation}

\subsection{Global Asymptotic Convergence}\label{Asymptotic_Convergence_section}
In order to better capture the dynamics of the variables, without loss of generality, we shift the equilibrium of the dynamics to zero by defining a set of new variables
\begin{equation}
    \begin{aligned}
        \Tilde{\xb} &\triangleq \xb - \xb^\star,\\
        \Tilde{\yb} &\triangleq \yb - \yb^\star,\\
        \Tilde{\zb} &\triangleq \zb - \zb^\star,
    \end{aligned}
\end{equation}
where $(\xb^\star,\yb^\star,\zb^\star)$ is the unique equilibrium point given in Lemma \ref{equilb_lemma}. We can then rewrite the first two equations of \eqref{setup_with_y} as follows
\begin{equation}
    \begin{aligned}
    \Dot{\Tilde{\zb}} &= -(\nabla f(\Tilde{\xb} + \xb^\star) + \mathbf{L} (\Tilde{\xb} + \xb^\star) + \Tilde{\yb} + \yb^\star)\label{setup_shift}\\
    &= -(\nabla f(\Tilde{\xb} + \xb^\star) - \nabla f(\xb^\star)) - \mathbf{L} \Tilde{\xb}- \Tilde{\yb} ,\\
    \Dot{\Tilde{\yb}} &=  \mathbf{L}\Tilde{\xb}, 
\end{aligned}
\end{equation}
where we used the fact that $\xb^\star = \1_n \otimes x^\star$ and $\yb^\star = - \nabla f(\xb^\star)$.

Now, as the matrix $\Lb = \Lc \otimes I_d$ is symmetric and positive semi-definite, there exists a decomposition $\Lb = \mathbf{Q\Lambda Q}^\top$, where $\Qb$ is an orthogonal matrix and $\mathbf{\Lambda} = \diag\{\lambda_1,\ldots, \lambda_{nd}\}$ is a diagonal matrix. Let 
$$\mathbf{S} = \Lb^{\frac{1}{2}}=\mathbf{Q\Lambda^{\frac{1}{2}} Q}^\top,$$ 
where $\mathbf{\Lambda}^\frac{1}{2} = \diag\{\sqrt{\lambda_1}, \ldots, \sqrt{\lambda_{nd}}\}$. 
Given \eqref{eq:32},
%Since $\mathbf{y}(0)=0$, $\dot{\yb}\in \text{image } \Lb$, we have $\yb(t)\in \text{image } \Lb$ and $\yb^\star\in \text{image } \Lb$, therefore 
there exists a variable $\wb(t)=\Sb\int_0^t\xb(\tau)d\tau$ and its centered version $\Tilde{\wb}=\wb-\wb^\star$, such that
\begin{equation}
    \begin{aligned}
        \yb &= \Sb \wb,~~~~ \text{and }~~~~\Tilde{\yb} = {\Sb}\Tilde{\wb}.\\
    \end{aligned}
\end{equation}
Replacing $\Tilde{\yb}$ in \eqref{setup_shift} with $\Sb\Tilde{\wb}$, we have that
\begin{equation}
    \begin{aligned}
    \Dot{\Tilde{\zb}}     &= -(\nabla f(\Tilde{\xb} + \xb^\star) - \nabla f(\xb^\star)) - \mathbf{L} \Tilde{\xb}-  \mathbf{S}\Tilde{\wb} ,\\
    \Dot{\Tilde{\wb}} &=  \mathbf{S}\Tilde{\xb} . \label{shift_S}
\end{aligned}
\end{equation}
Following the proof of Lemma \ref{equilb_lemma}, it is straightforward to show that the dynamics above at equilibrium satisfies $\Tilde{\xb}=\0$.

\begin{theorem}\label{theorem1} 
    Given Assumptions \ref{A1}-\ref{Assumption_phi}, for any starting point $x_i(0) = x_{i0}, z_i(0) = z_{i0}$ with $x_{i0}=\nabla\phi^\star(z_{i0})$, the distributed mirror descent algorithm with integral feedback proposed in \eqref{setup} will converge to the global optimum asymptotically, i.e., $\lim_{t\to\infty} x_i(t)=x^\star$ for any $i\in [n]$.
\end{theorem}
\begin{proof}
We study the convergence of the dynamics \eqref{shift_S}. Let us consider the candidate Lyapunov function
\begin{equation}\label{lyap}
    V(\Tilde{\zb},\Tilde{\wb}) = \sum_{i=1}^n \Dc_{\phi^\star}(z_i, z^\star) + \frac{1}{2} \Tilde{\wb}^\top   \Tilde{\wb}.
\end{equation}
Notice that the Bregman divergence used in the candidate Lyapunov function is defined with respect to $\phi^\star$.
Since $\phi^\star$ is convex and $\mu_\phi^{-1}$-smooth, $V(\Tilde{\zb},\Tilde{\wb})$ is non-negative and has Lipschitz-continuous first derivatives. Differentiating $V$ and recalling \eqref{setup}-\eqref{setup_shift}-\eqref{shift_S}, we derive

\begin{small}
\begin{align*}
    \dot{V} &= \frac{d}{dt}\sum_{i=1}^n \bigg(\phi^\star(z_i) -\phi^\star(z^\star) - \langle \nabla\phi^\star(z^\star), z_i - z^\star \rangle \bigg)  +\Tilde{\mathbf{w}}^\top \frac{d \Tilde{\mathbf{w}} }{dt}\\
    &= \sum_{i=1}^n \langle x_i -x^\star, \frac{d z_i}{dt}\rangle + \Tilde{\mathbf{w}}^\top S \Tilde{\mathbf{x}}\\
    &= \langle \mathbf{x} -\mathbf{x}^\star, \frac{d \mathbf{z}}{dt}\rangle+ \Tilde{\mathbf{w}}^\top S \Tilde{\mathbf{x}}\\
    &= \vphantom{\frac{d}{dt}} \langle \Tilde{\mathbf{x}}, -(\nabla f(\Tilde{\mathbf{x}}+ \mathbf{x}^\star) - \nabla f( \mathbf{x}^\star) + \mathbf{L} \Tilde{\mathbf{x}} + S\Tilde{\mathbf{w}})\rangle+ \Tilde{\mathbf{w}}^\top S \Tilde{\mathbf{x}}\\
    &= \vphantom{\frac{d}{dt}}- \langle \Tilde{\mathbf{x}}, \nabla f(\Tilde{\mathbf{x}}+ \mathbf{x}^\star) - \nabla f( \mathbf{x}^\star) \rangle -\Tilde{\mathbf{x}}^\top \mathbf{L} \Tilde{\mathbf{x}}.
\end{align*}
\end{small}
It is clear from the convexity of local functions that $\dot{V}\leq 0$ at all times. When local variables do not have consensus, $\dot{V}<0$ since $\Tilde{\mathbf{x}}^\top \mathbf{L} \Tilde{\mathbf{x}} < 0$. When consensus is reached, $\Tilde{\xb}=\1_n \otimes a$, and then the first term $\langle \Tilde{\mathbf{x}}, \nabla f(\Tilde{\mathbf{x}}+ \mathbf{x}^\star) - \nabla f( \mathbf{x}^\star) \rangle = \sum_{i=1}^n \langle a, \nabla f_i(a+ {x}^\star) - \nabla f_i( {x}^\star) \rangle$ is equal to zero if and only if $a=\0$, which implies $\Tilde{\xb}=\0$. The uniqueness of $a=\0$ is due to the strong convexity in Assumption \ref{A2}. Therefore, the condition for equality is $\xb=\xb^\star$, which also gives $\zb = \zb^\star$.

The Lyapunov function satisfies $V>0$ for $\Tilde{\zb} \neq \0$, and $V=0$ when $\Tilde{\zb} =\0,\Tilde{\wb}=\0$. We also have $\dot{V} \leq 0$ with equality only at equilibrium. Then, by LaSalle's invariance principle, the dynamics \eqref{shift_S} will converge asymptotically to its equilibrium point, and this completes the proof.
\end{proof}

\section{Numerical Simulation}\label{simulation}
In this section, we first derive a discretized version of our algorithm in \eqref{discrete_final} and then illustrate a numerical example that shows the advantage of using integral feedback for speeding up the convergence of distributed mirror descent. 
\subsection{Discretization}
Recall the continuous-time dynamics \eqref{setup_with_y}. We use Euler's method to derive a discrete version of the algorithm as follows. We first choose a time interval for discretization denoted by $\Delta t$. Let $t_k \triangleq k \Delta t$ and  $\xb^{(k)} \triangleq \xb(t_k) = \xb(k\Delta t)$. We can similarly define $\yb^{(k)}$ and $\zb^{(k)}$. We then have the following discrete updates
\begin{equation}
    \begin{aligned}\label{discrete_with_y1}
    \frac{\mathbf{z}^{(k+1)}-\mathbf{z}^{(k)}}{\Delta t}&=  -(\nabla f(\mathbf{x}^{(k)}) + \mathbf{L} \mathbf{x}^{(k)} + \mathbf{y}^{(k)}),\\
   \frac{\mathbf{y}^{(k+1)}-\mathbf{y}^{(k)}}{\Delta t} &=  \mathbf{L} \mathbf{x}^{(k)}. 
\end{aligned}
\end{equation}

After re-arranging the terms, the fully distributed mirror descent algorithm with integral feedback takes the following (discrete) form
\begin{equation}\label{discrete_final}
    \begin{aligned}
        {z_i}^{(k+1)}&=  {z_i}^{(k)} - \bigg( \nabla f_i({x_i}^{(k)}) + {y_i}^{(k)}\vphantom{\sum_{j \in \Nc_i}} \\
    & \qquad  +\sum_{j \in \Nc_i} ({x_i}^{(k)} - {x_j}^{(k)})  \bigg)\Delta t,
     \\
   {y_i}^{(k+1)} &= {y_i}^{(k)} + \sum_{j \in \Nc_i} ({x_i}^{(k)} - {x_j}^{(k)}) \Delta t \vphantom{\bigg)},\\
   {x_i}^{(k+1)} &= \nabla\phi^\star({z_i}^{(k+1)}) \vphantom{\sum_{j \in \Nc_i}\bigg)}.
\end{aligned}
\end{equation}

\subsection{Numerical Example}
We now provide a simulation for the update  \eqref{discrete_final}. 

{\bf -- Network Structure:} We consider a $10$-agent cycle network, where each agent is connected to its previous and next agent, and the last agent is connected to the first agent. The network structure is shown in Fig. \ref{network}. 

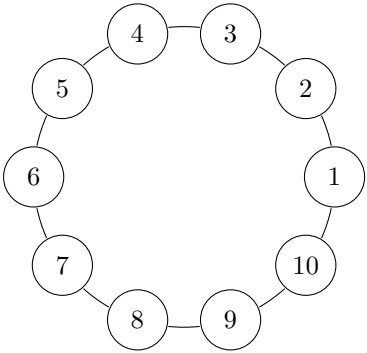
\begin{figure}[h]
\centering
\begin{tikzpicture}

\def \n {10}
\def \radius {2cm}
\def \margin {12} % margin in angles, depends on the radius
    % draw/.style={circle, draw, minimum size=0.2cm}

\foreach \s in {1,...,\n}
{
  \node[draw, circle, minimum size=0.8cm] at ({360/\n * (\s - 1)}:\radius) {$\s$};
  \draw[-, >=latex] ({360/\n * (\s - 1)+\margin}:\radius) 
    arc ({360/\n * (\s - 1)+\margin}:{360/\n * (\s)-\margin}:\radius);
}
\end{tikzpicture}
\caption{Structure of the undirected network}
\label{network}
\end{figure}

{\bf -- Generating Function for Mirror Descent:} To implement mirror descent, we employ the commonly used negative entropy as the generating function, where
\begin{equation*}
        \phi(x) = \sum_{j=1}^d [x]_j \log([x]_j) \Longrightarrow
        [z]_i = [\nabla \phi(x)]_i = 1+ \log([x]_i),
\end{equation*}
and by convention $[x]_j \log([x]_j)=0$ if $[x]_j=0$. Here, $x$ and $z$ are both $d$-dimensional vectors. By simple calculations, it can be shown that $\phi^\star$, the convex conjugate of $\phi$, takes the following form
\begin{equation*}
        \phi^\star(z) = \sum_{j=1}^d e^{[z]_j -1} \Longrightarrow
        [x]_i = [\nabla \phi^\star(z)]_i = e^{[z]_i -1}.
\end{equation*}
Thus, we can now implement \eqref{discrete_final}.

{\bf -- Global and Local Functions:} To construct the functions, we first generate a $100$-dimensional vector $u$ following a Gaussian distribution $\Nc(10\times\1_d, I_d)$. We then perturb $u$ to generate local optima $u_i = u+w_i$, where $w_i\sim \Nc(\0, I_d)$ for $i\in[n]$. We set the local functions $f_i(x) = \frac{1}{2}\norm{A_i x - b_i}^2$, where $b_i = A_i u_i$ and $A_i \in \R^{20\times 100}$ is a random matrix of rank $15$. The global function becomes $F(x) = \frac{1}{2}\norm{\Ab x - \bb}^2$, where $\Ab\in \R^{200\times 100}$ and $\bb\in \R^{200}$ are stacked versions of their distributed counterparts. We can verify that $F$ is strongly convex, and the closed-form solution for this problem is $x^\star = \Ab^\dagger \bb$. We run \eqref{discrete_final} with a feasible random initialization $x_i^{(0)}$, and let $y_i^{(0)} = \0$ and $z_i^{(0)} = \nabla \phi (x_i^{(0)})$ for every $i\in [n]$. Recall that $n=10$ and $d=100$, and we set $\Delta t = 10^{-2}.$ 

Note that the local objective functions are only convex, but the global objective function $F(x) = \sum_{i=1}^n f_i(x)$ is strongly convex in consistent with our theoretical assumptions. 

{\bf -- Performance:} We compare our method with distributed mirror descent {\it without} integral feedback \cite{li2016distributed,shahrampour2017distributed}. These works were originally proposed for convex global functions with a suggested diminishing step-size $\frac{1}{\sqrt{k}}$. Beside that, we also include their performance with constant step-size.

\begin{figure}[h]
\centering
\includegraphics[width=0.5\textwidth]{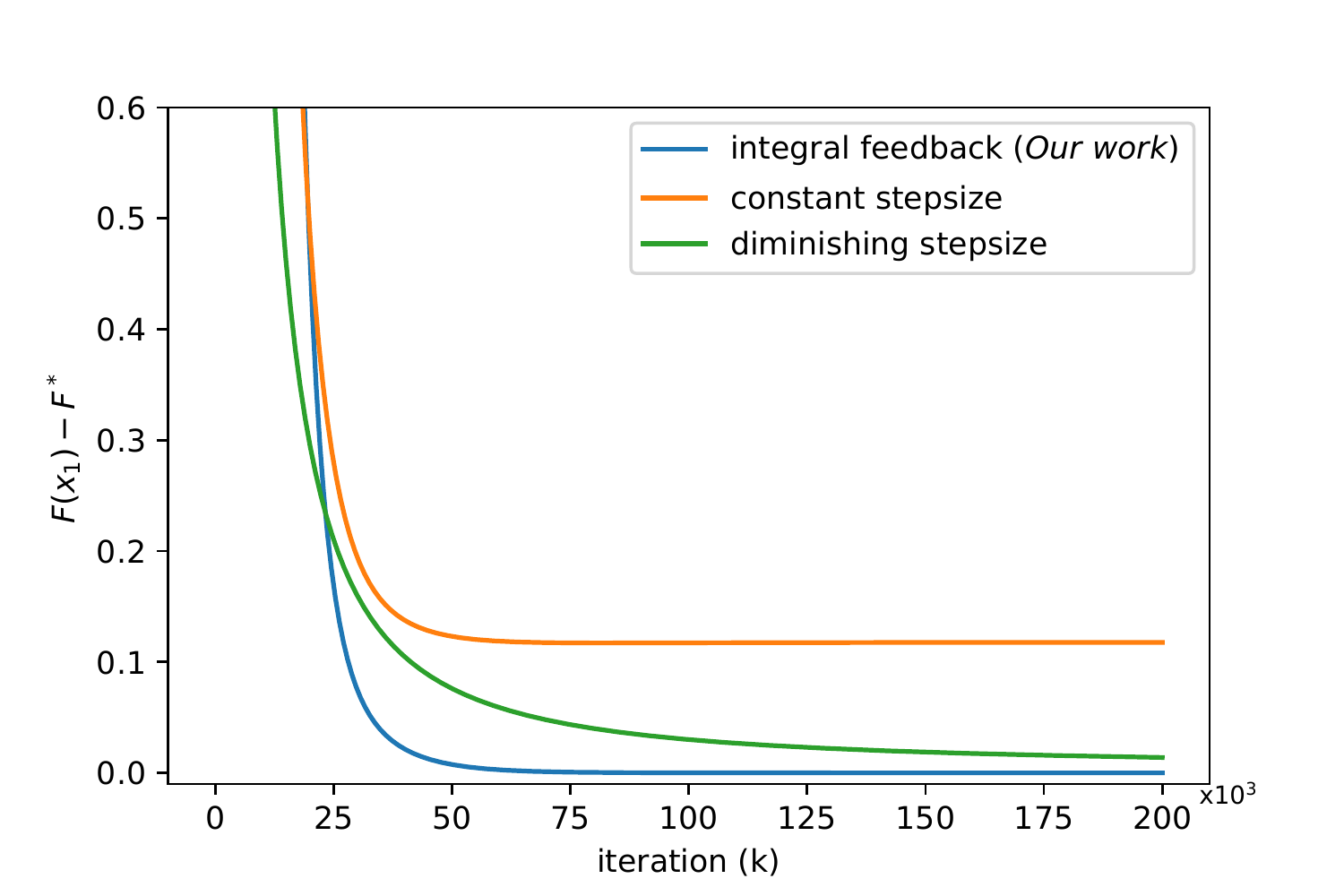}
\caption{The trajectory of the global objective evaluated at agent $1$}
\label{comparison}

\end{figure}
\begin{figure}[h]
\centering
\includegraphics[width=0.5\textwidth]{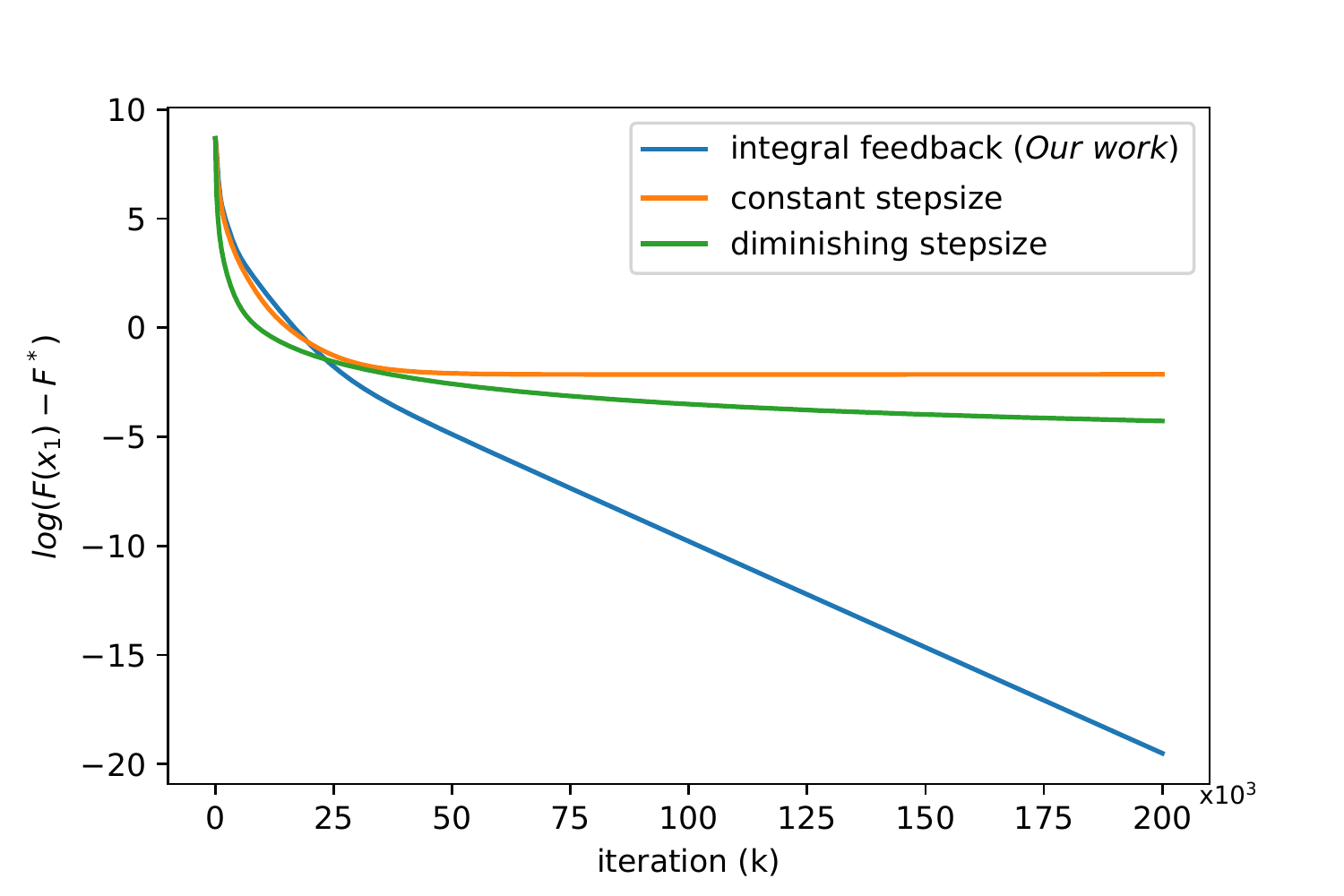}
\caption{The trajectory of the log-distance to global solution evaluated at agent $1$}
\label{log_comparison}

\end{figure}

For all three algorithms, we plot $F(x_1^{(k)})- F^\star$ with respect to iteration $k$ in Fig. \ref{comparison}, which presents the convergence properties of agent $1$. We can see that our method converges faster that distributed mirror descent without integral feedback. In fact, without integral feedback, agents never converge to the global solution using a constant step-size, because the local objective functions have different local minima. We further plot $\log(F(x_1^{(k)})- F^\star)$ with respect to iteration $k$ in Fig. \ref{log_comparison}. Interestingly, our method exhibits a linear convergence rate (i.e., exponentially fast), which is on par with the state-of-the-art distributed gradient descent methods (in the sense of achieving a linear rate). We reiterate that diminishing step-size is suitable for convex (and not strongly convex) global objective functions. The main purpose of the comparisons with other methods is to illustrate the power of integral feedback as soon as the strong convexity assumption is satisfied.

\section{Conclusion}\label{Conclusion}
In this paper, we considered a distributed optimization scenario where a network of agents aims at minimizing a strongly convex function, that can be written as a sum of local convex functions. The agents only have access to local gradients, but they are able to exchange information with one other. We proposed a fully decentralized mirror descent algorithm that enforces consensus among agents through a consensus term plus an additional integral feedback. We studied the continuous-time dynamics of the algorithm and provided asymptotic convergence using Lyapunov stability. Focusing on strongly convex problems, we presented empirical results verifying that distributed mirror descent with integral feedback enjoys a faster convergence rate, compared to its variants without integral feedback.

This paper provides technical analysis for the asymptotic convergence, but the simulations show that the algorithm (perhaps with smoothness assumption) can exhibit exponential convergence. Therefore, a potential future direction is the theoretical analysis of this behavior. Furthermore, studying the theoretical guarantees of \eqref{discrete_final} will shed more light on required technical assumptions in maintaining the same convergence rate when transitioning from the continuous-time update to the discrete-time update.

\bibliographystyle{IEEEtran}
\bibliography{References}

\begin{thebibliography}{10}
\providecommand{\url}[1]{#1}
\csname url@samestyle\endcsname
\providecommand{\newblock}{\relax}
\providecommand{\bibinfo}[2]{#2}
\providecommand{\BIBentrySTDinterwordspacing}{\spaceskip=0pt\relax}
\providecommand{\BIBentryALTinterwordstretchfactor}{4}
\providecommand{\BIBentryALTinterwordspacing}{\spaceskip=\fontdimen2\font plus
\BIBentryALTinterwordstretchfactor\fontdimen3\font minus
  \fontdimen4\font\relax}
\providecommand{\BIBforeignlanguage}[2]{{%
\expandafter\ifx\csname l@#1\endcsname\relax
\typeout{** WARNING: IEEEtran.bst: No hyphenation pattern has been}%
\typeout{** loaded for the language `#1'. Using the pattern for}%
\typeout{** the default language instead.}%
\else
\language=\csname l@#1\endcsname
\fi
#2}}
\providecommand{\BIBdecl}{\relax}
\BIBdecl

\bibitem{nemirovsky1983originalmd}
A.~S. Nemirovsky and D.~B. Yudin, ``Problem complexity and method efficiency in
  optimization.'' 1983.

\bibitem{ben2001ordered}
A.~Ben-Tal, T.~Margalit, and A.~Nemirovski, ``The ordered subsets mirror
  descent optimization method with applications to tomography,'' \emph{SIAM
  Journal on Optimization}, vol.~12, no.~1, pp. 79--108, 2001.

\bibitem{bloch1994hamiltonian}
A.~Bloch, \emph{Hamiltonian and gradient flows, algorithms and control}.\hskip
  1em plus 0.5em minus 0.4em\relax American Mathematical Soc., 1994, vol.~3.

\bibitem{brown1989some}
A.~Brown and M.~C. Bartholomew-Biggs, ``Some effective methods for
  unconstrained optimization based on the solution of systems of ordinary
  differential equations,'' \emph{Journal of Optimization Theory and
  Applications}, vol.~62, no.~2, pp. 211--224, 1989.

\bibitem{helmke2012optimization}
U.~Helmke and J.~B. Moore, \emph{Optimization and dynamical systems}.\hskip 1em
  plus 0.5em minus 0.4em\relax Springer Science \& Business Media, 2012.

\bibitem{krichene2015acceleratedmd}
W.~Krichene, A.~Bayen, and P.~L. Bartlett, ``Accelerated mirror descent in
  continuous and discrete time,'' in \emph{Advances in Neural Information
  Processing Systems (NeurIPS)}, 2015, pp. 2845--2853.

\bibitem{li2016distributed}
J.~Li, G.~Chen, Z.~Dong, and Z.~Wu, ``Distributed mirror descent method for
  multi-agent optimization with delay,'' \emph{Neurocomputing}, vol. 177, pp.
  643--650, 2016.

\bibitem{shahrampour2017distributed}
S.~Shahrampour and A.~Jadbabaie, ``Distributed online optimization in dynamic
  environments using mirror descent,'' \emph{IEEE Transactions on Automatic
  Control}, vol.~63, no.~3, pp. 714--725, 2018.

\bibitem{nedic2009distributed}
A.~Nedic and A.~Ozdaglar, ``Distributed subgradient methods for multi-agent
  optimization,'' \emph{IEEE Transactions on Automatic Control}, vol.~54,
  no.~1, pp. 48--61, 2009.

\bibitem{shi2015extra}
W.~Shi, Q.~Ling, G.~Wu, and W.~Yin, ``Extra: An exact first-order algorithm for
  decentralized consensus optimization,'' \emph{SIAM Journal on Optimization},
  vol.~25, no.~2, pp. 944--966, 2015.

\bibitem{sun2019convergence}
Y.~Sun, A.~Daneshmand, and G.~Scutari, ``Convergence rate of distributed
  optimization algorithms based on gradient tracking,'' \emph{arXiv preprint
  arXiv:1905.02637}, 2019.

\bibitem{di2016next}
P.~Di~Lorenzo and G.~Scutari, ``Next: In-network nonconvex optimization,''
  \emph{IEEE Transactions on Signal and Information Processing over Networks},
  vol.~2, no.~2, pp. 120--136, 2016.

\bibitem{lin2019distributed}
P.~Lin, W.~Ren, C.~Yang, and W.~Gui, ``Distributed continuous-time and
  discrete-time optimization with nonuniform unbounded convex constraint sets
  and nonuniform stepsizes,'' \emph{IEEE Transactions on Automatic Control},
  vol.~64, no.~12, pp. 5148--5155, 2019.

\bibitem{liu2014continuous}
S.~Liu, Z.~Qiu, and L.~Xie, ``Continuous-time distributed convex optimization
  with set constraints,'' \emph{IFAC Proceedings Volumes}, vol.~47, no.~3, pp.
  9762--9767, 2014.

\bibitem{cortes2013weightbalanceddigraph}
B.~Gharesifard and J.~Cort{\'e}s, ``Distributed continuous-time convex
  optimization on weight-balanced digraphs,'' \emph{IEEE Transactions on
  Automatic Control}, vol.~59, no.~3, pp. 781--786, 2013.

\bibitem{7744584}
X.~{Zeng}, P.~{Yi}, and Y.~{Hong}, ``Distributed continuous-time algorithm for
  constrained convex optimizations via nonsmooth analysis approach,''
  \emph{IEEE Transactions on Automatic Control}, vol.~62, no.~10, pp.
  5227--5233, 2017.

\bibitem{kia2015distributed}
S.~S. Kia, J.~Cort{\'e}s, and S.~Mart{\'\i}nez, ``Distributed convex
  optimization via continuous-time coordination algorithms with discrete-time
  communication,'' \emph{Automatica}, vol.~55, pp. 254--264, 2015.

\bibitem{yang2016multi}
S.~Yang, Q.~Liu, and J.~Wang, ``A multi-agent system with a
  proportional-integral protocol for distributed constrained optimization,''
  \emph{IEEE Transactions on Automatic Control}, vol.~62, no.~7, pp.
  3461--3467, 2016.

\bibitem{9070199}
D.~{Yuan}, Y.~{Hong}, D.~W.~C. {Ho}, and S.~{Xu}, ``Distributed mirror descent
  for online composite optimization,'' \emph{IEEE Transactions on Automatic
  Control}, pp. 1--1, 2020.

\bibitem{yuan2018optimal}
D.~Yuan, Y.~Hong, D.~W. Ho, and G.~Jiang, ``Optimal distributed stochastic
  mirror descent for strongly convex optimization,'' \emph{Automatica},
  vol.~90, pp. 196--203, 2018.

\bibitem{7383850}
M.~{Rabbat}, ``Multi-agent mirror descent for decentralized stochastic
  optimization,'' in \emph{IEEE 6th International Workshop on Computational
  Advances in Multi-Sensor Adaptive Processing (CAMSAP)}, 2015, pp. 517--520.

\bibitem{shahrampour2015distributed}
S.~Shahrampour, A.~Rakhlin, and A.~Jadbabaie, ``Distributed detection:
  Finite-time analysis and impact of network topology,'' \emph{IEEE
  Transactions on Automatic Control}, vol.~61, no.~11, pp. 3256--3268, 2016.

\bibitem{8409957}
T.~T. {Doan}, S.~{Bose}, D.~H. {Nguyen}, and C.~L. {Beck}, ``Convergence of the
  iterates in mirror descent methods,'' \emph{IEEE Control Systems Letters},
  vol.~3, no.~1, pp. 114--119, 2019.

\bibitem{borovykh2020interact}
A.~Borovykh, N.~Kantas, P.~Parpas, and G.~A. Pavliotis, ``To interact or not?
  the convergence properties of interacting stochastic mirror descent,'' in
  \emph{International Conference on Machine Learning (ICML) Workshop on
  ‘Beyond First order methods in ML Systems}, 2020.

\bibitem{raginsky2012Variancereduction}
M.~Raginsky and J.~Bouvrie, ``Continuous-time stochastic mirror descent on a
  network: Variance reduction, consensus, convergence,'' in \emph{IEEE
  Conference on Decision and Control (CDC)}, 2012, pp. 6793--6800.

\bibitem{hiriart2012fundamentals}
J.-B. Hiriart-Urruty and C.~Lemar{\'e}chal, \emph{Fundamentals of convex
  analysis}.\hskip 1em plus 0.5em minus 0.4em\relax Springer Science \&
  Business Media, 2012.

\end{thebibliography}

%%%%%%%%%%%%%%%%%%%%%%%%%%%%%%%%%%%%%%%%%%%%%%%%%%%%%%%%%%%%%%%%%%%%%%%%%%%%%%%%

\end{document}